\documentclass[10pt,a4paper,twoside]{amsart}
\usepackage[T1]{fontenc}
\usepackage[utf8]{inputenc}
\usepackage[english]{babel}
\usepackage{amsmath}
\usepackage{amsthm}
\usepackage{amssymb}
\usepackage{amsfonts}
\usepackage{amsxtra}
\usepackage{enumerate}
\usepackage{verbatim}
\usepackage{color}
\usepackage[mathscr]{eucal}
\usepackage{graphicx}
\usepackage{xcolor}
\usepackage[backref=page]{hyperref}
\usepackage{tikz}
\usepackage{enumitem}
\usepackage{parskip}

\newcommand{\R}{\mathbb R}
\newcommand{\C}{\mathcal C}

\newcommand{\D}{\mathcal{D}}

\newtheorem{theorem}{Theorem}[section]

\newtheorem{lemma}[theorem]{Lemma}
\newtheorem{cor}[theorem]{Corollary}

\newtheorem{prop}[theorem]{Proposition}

\newtheorem{remark}[theorem]{Remark}

\theoremstyle{definition}

\DeclareMathOperator*{\cont}{Cont}
\DeclareMathOperator*{\diff}{Diff}

\let\phi=\varphi

\newcommand{\rem}[1]{}

\DeclareFontFamily{U}{mathb}{\hyphenchar\font45}
\DeclareFontShape{U}{mathb}{m}{n}{
<-6> mathb5 <6-7> mathb6 <7-8> mathb7
<8-9> mathb8 <9-10> mathb9
<10-12> mathb10 <12-> mathb12
}{}
\DeclareSymbolFont{mathb}{U}{mathb}{m}{n}
\DeclareMathSymbol{\llcurly}{\mathrel}{mathb}{"CE}
\DeclareMathSymbol{\ggcurly}{\mathrel}{mathb}{"CF}

\newcommand{\id}{\mathrm{id}}

\title[Positive paths in Diffeomorphism groups]{Positive paths in diffeomorphism groups of manifolds with a contact distribution}

\author{Jakob Hedicke}
\address{ Radboud Universiteit Nijmegen, Heyendaalseweg 135, 6525 AJ NIJMEGEN, The Netherlands} 
\email{jakob.hedicke@gmail.com}

\date{\today}
\begin{document}

\begin{abstract}
Given a cooriented contact manifold $(M,\xi)$, it is possible to define a notion of positivity on the group $\mathrm{Diff}(M)$ of diffeomorphisms of $M$, by looking at paths of diffeomorphisms that are positively transverse to the contact distribution $\xi$.
We show that, in contrast to the analogous notion usually considered on the group of diffeomorphisms preserving $\xi$, positivity on $\mathrm{Diff}(M)$ is completely flexible.
In particular, we show that for the standard contact structure on $\R^{2n+1}$ any two diffeomorphisms are connected by a positive path.
This result generalizes to compactly supported diffeomorphisms on a large class of contact manifolds.
As an application we answer a question about Legendrians in thermodynamic phase space posed by Entov, Polterovich and Ryzhik in the context of thermodynamic processes.
\end{abstract} 

\maketitle

\section{Introduction}

Let $(M,\xi)$ be a (not necessarily closed) cooriented contact manifold of dimension $2n+1$ for $n\geq 1$.
This means that the hyperplane distribution $\xi\subset TM$ is globally the kernel of a $1$-form $\alpha$ such that $\alpha\wedge (d\alpha)^n$ is a volume form on $M$.

There are two natural groups associated with $M$: the group of diffeomorphisms $\diff(M)$ and the group of contactomorphisms $\cont(M,\xi)$, i.e., the subgroup of diffeomorphisms preserving the contact structure $\xi$.
In this paper we will mostly consider the following subgroups of $\diff(M)$: the group $\diff_0(M)$ of diffeomorphisms isotopic to the identity, the group $\diff_c(M)$ of compactly supported diffeomorphisms isotopic to the identity via compactly supported diffeomorphisms and their contact analogues $\cont_0(M,\xi)$ and $\cont_c(M,\xi)$.

A fundamental goal in contact topology is to study rigidity phenomena for contactomorphisms, or in other words to understand the differences between the groups $\diff(M)$ and $\cont(M,\xi)$.

In this short note we consider a notion of positivity on $\diff(M)$ induced by the contact structure, analogously to the notion considered in contact topology \cite{Eliashberg00}.
While in contact geometry this notion sometimes gives rise to a partial order, we show that all diffeomorphism groups are non-orderable.
More precisely, the relation induced by positivity is trivial for compactly supported diffeomorphisms, and there exist $C^{\infty}$-small positive loops of diffeomorphisms.
This is a big contrast to the group of contactomorphisms, where even in the non-orderable case small positive loops do not exist \cite{Casals162, Albers17}.
Our main tool is a Chow-Rashevsky Theorem for diffeomorphism groups, proved in \cite{Agrachev09}. 

For the case of Legendrians we use these results to answer a question posed in \cite{Entov25} in the context of thermodynamic processes.

\subsection{Positivity for contactomorphisms and diffeomorphisms}

In \cite{Eliashberg00} the authors introduce the notion of a positive path of contactomorphisms (see also \cite{Bhupal01}) as follows. 
Any smooth path $(\phi_t)_{t\in I}$ in $\cont(M,\xi)$ defines a time-dependent contact vector field $(X_t^{\phi})_{t\in I}$ on $M$.
A path is called \textbf{positive (non-negative)} if $\alpha \left(X_t^{\phi}\right)>0$ ($\geq 0$) for all $t$, where $\alpha$ is any contact form defining the chosen coorientation of $\xi$.
The notion of positivity can be used to define a bi-invariant relation on $\cont(M,\xi)$ (and its universal cover $\widetilde{\cont(M,\xi)}$), see e.g. \cite{Eliashberg00}.
For closed contact manifolds this relation turns out to be a partial order if and only if there are no positive (contractible) loops of contactomorphisms \cite{Eliashberg00}.

The notion of positivity has turned out to be connected to several other known phenomena in contact geometry, such as for example the existence of closed Reeb orbits \cite{Albers15}, contact (non-)squeezing \cite{Eliashberg06, Albers18}, the existence of invariant norms \cite{Colin15, Fraser18}, translated points \cite{Sandon12, Albers15, Arlove25, Hedicke24}, the existence of quasimorphisms \cite{Givental,Granja}, (non-) displaceability of pre-Lagrangians \cite{Borman152, Marincovic16}, loose Legendrians \cite{Liu20} or supporting open book decompositions \cite{Hedicke24}.

In principle the same notion of positivity can be defined for paths in $\diff(M)$, see \cite[Remark 5.4]{Entov25}.
Since any path $(f_t)_{t\in I}$ defines a time-dependent vector field $(X_t^f)_{t\in I}$ on $M$, we can analogously define a path of diffeomorphisms to be positive (non-negative) if $\alpha \left(X_t^{f}\right)>0$ ($\geq 0$) for all $t$.
Note however that this notion of positivity is in general not invariant under conjugation.
Another major difference to the contact case is that there exist non-constant paths of diffeomorphisms with $\alpha \left(X_t^{f}\right)=0$.
We call a path $(f_t)_{t\in I}$ with this property a \textbf{null path}.

In this paper we show that the notion of positivity is trivial on the identity component of the diffeomorphism group $\diff_0(M)$ in the sense of the following theorems.

\begin{theorem}\label{thm:Rn:all}
Consider $\R^{2n+1}$ equipped with its standard contact structure.
Every $f\in \diff_0(\R^{2n+1})$ is connected to $\mathrm{id}$ by a null path and a positive path of diffeomorphisms.
If $f\in\diff_c(\R^{2n+1})$, it is connected to $\mathrm{id}$ by a compactly supported null path and a positive path that is arbitrarily close to $\mathrm{id}$ in the $C^{\infty}$-compact open topology outside of a compact set.
\end{theorem}

In this special case of the standard contact structure on $\R^{2n+1}$, the result generalizes a Chow-Rashevsky Theorem for diffeomorphism groups proved in \cite{Agrachev09} for compact manifolds.

The following shows that in $\R^{2n+1}$ it is possible to deform a given path of contactomorphisms that is positive outside a compact subset to an everywhere positive path of diffeomorphisms with the same endpoints.

\begin{theorem}\label{thm:extension}
Consider $\R^{2n+1}$ equipped with its standard contact structure.
Let $(f_t)_{t\in[0,1]}$ be a path of contactomorphisms that is positive on the complement of some compact subset $K_1$.
Then there exists a positive path of diffeomorphisms $(g_t)_{t\in [0,1]}$ connecting $f_0$ and $f_1$ such that $g_t$ coincides with $f_t$ on the complement of some compact subset $K_2\supset K_1$. 
\end{theorem}

For more general contact manifolds one can show the following.

\begin{theorem}\label{thm:path}
Let $(M,\xi)$ be a contact manifold and $(f_t)_{t\in[0,1]}$ be a compactly supported path of diffeomorphisms.
Suppose $(M,\xi)$ can be covered by finitely many Darboux charts.
Then there exist a compactly supported null path with the same endpoints and a positive path (not necessarily compactly supported) with the same endpoints.
\end{theorem}

For the case of compactly supported diffeomorphisms the local contractibility of the $C^{\infty}$-topology implies the following.

\begin{theorem}\label{thm:homotopic}
Any path of compactly supported diffeomorphisms $(f_t)_{t\in[0,1]}$ is homotopic in $\diff_c(M)$ with fixed endpoints to a compactly supported null path.
If $M$ is closed, any path of diffeomorphisms is homotopic with fixed endpoints to a positive path.
Moreover, there exist $C^{\infty}$-small positive and null loops in $\diff_0(M)$.
\end{theorem}

On many closed contact manifolds it is known that a contractible positive loop of contactomorphisms can never be contracted in $\cont_0(M,\xi)$ via non-negative loops, see e.g. \cite[Theorem 1.11]{Eliashberg06} for the case of $S^{2n+1}$.
This changes for diffeomorphisms.

\begin{theorem}\label{cor:positive:contraction}
Let $(M,\xi)$ be a closed contact manifold.
Then there exists a homotopy of loops $H\colon S^1\times [0,1]\rightarrow\diff_0(M)$ such that $H(t,1)=\mathrm{id}$ for all $t\in S^1$ and $t\mapsto H(t,s)$ is a positive loop for all $s\in [0,1)$.
\end{theorem}

\subsection{Positive paths between submanifolds and thermodynamic processes}

As pointed out in \cite[Remark 5.4]{Entov25} one can also define a notion of positivity for smooth families of embeddings into a contact manifold.
Let $N$ be a smooth manifold and $\iota\colon N\times[0,1]\hookrightarrow M$ be a smooth family of embeddings into a cooriented contact manifold $(M,\xi=\ker\alpha)$.
We assume that all embeddings are proper in the sense that the intersection of the image with any compact subset in $M$ is compact.
Call $\iota$ \textbf{positive/ non-negative/ null} if $\alpha(d\iota(\partial_t))$ is positive/ non-negative/ null for any contact form $\alpha$ inducing the given coorientation, i.e., if the path $t\mapsto\iota(p,t)$ is positively/non-negatively transverse or tangent to $\xi$ for any $p\in N$.

As a direct consequence of Theorem \ref{thm:path} and the isotopy extension theorem \cite{Hirsch76} we get.

\begin{cor}\label{thm:Leg}
Let $(M,\xi)$ be a contact manifold and $N_1,N_2$ be two submanifolds that are isotopic via a compactly supported isotopy.
Assume that $(M,\xi)$ can be covered by finitely many Darboux charts.
Then there exist a positive isotopy and a compactly supported null isotopy between $N_1$ and $N_2$.
\end{cor}

In \cite{Entov25} the authors look at positive (non-negative) isotopies of Legendrians in the context of thermodynamic processes.
More precisely, they look at sets of equilibrium states in a thermodynamic phase space $\mathcal{T}$.
The space $\mathcal{T}$ naturally carries a contact structure and is contactomorphic to the $1$-jet bundle $J^1\R^n\cong \R^{2n+1}$ of $\R^n$ (or more generally of any smooth manifold $X$) equipped with its canonical contact structure.
Sets of equilibrium states naturally form Legendrian submanifolds in this space, i.e., submanifolds $L$ with $TL\subset \xi$.
As before we assume that $\mathcal{T}$ can be covered by finitely many Darboux charts.
Slow thermodynamic processes can be described in this language by compactly supported non-negative isotopies of Legendrians, see \cite[Section 5]{Entov25}.
In \cite[Remark 5.4]{Entov25} the authors consider smooth (possibly non-Legendrian) isotopies between two Legendrians and raise the question if the non-negativity assumption is restrictive in this setting.
The following corollary of Theorem \ref{thm:Rn:all} and Theorem \ref{thm:path} provides a negative answer.

\begin{cor}
Let $L_1, L_2$ be Legendrians in $\mathcal{T}$ that are isotopic via a compactly supported smooth isotopy.
Then there exist a positive and a compactly supported null isotopy between $L_1$ and $L_2$.
If $\mathcal{T}\cong J^1\R^n$ the statement is true for any two Legendrians $L_1,L_2$ such that there exist $f\in\diff_0(J^1\R^n)$ with $f(L_1)=L_2$.
Moreover, in this case the positive isotopy can be chosen arbitrarily close in the $C^{\infty}$-compact open topology to the constant isotopy in the complement of a compact subset.
\end{cor}

\begin{remark}
The results follow from the isotopy extension theorem for submanifolds related by a smooth compactly supported isotopy \cite[Chapter 8]{Hirsch76}.
However, for general (non-compactly supported) smoothly isotopic submanifolds it does in general not seem to be known if there exists an isotopy between them that is induced by a path of diffeomorphisms, which is needed in order to apply Theorem \ref{thm:path}.
\end{remark}

\begin{cor}
Let $\iota\colon L\times[0,1]\hookrightarrow J^1\R^n$ be a Legendrian isotopy in $J^1\R^n$ that is induced by a path of contactomorphisms and that is positive outside a compact subset.
Then there exists a positive smooth isotopy $\hat{\iota}\colon L\times[0,1]\hookrightarrow J^1\R^n$ between $L_0$ and $L_1$ that coincides with $\iota$ in the complement of a compact subset.
\end{cor}

\subsection{Norms on groups of diffeomorphisms}

A further direct consequence of Theorem \ref{thm:path} and Corollary \ref{thm:Leg} is that, in contrast to contact geometry \cite{Shelukhin17, Rosen20}, the notion of contact Hamiltonian can not be used to define any meaningful Hofer-type (psudo-) norms or metrics on the group $\diff_c(M)$ or the isotopy class of any submanifold.

For instance given a path of diffeomorphisms $(f_t)_{t\in[0,1]}$ on a closed manifold $M$ with generating time-dependent vector field $X_t$, one can define its Hofer length with respect to a contact form $\alpha$ as
$$l_{\alpha}(f_t):=\int\limits_0^1\max_M|\alpha(X_t)|dt.$$
Taking the infimum of this length over all paths with the same endpoint defines a norm on $\cont_0(M,\xi)$ \cite{Shelukhin17}.

It follows from Theorem \ref{thm:path} that this infimum vanishes on $\diff_0(M)$.

\subsection*{Acknowledgements}

I thank Leonid Polterovich for his comments on a first version of the paper and for encouraging to prove parts of Theorem \ref{thm:Rn:all} and Theorem \ref{thm:extension}.
I also thank Egor Shelukhin, Souheib Allout and the anonymous referee for valuable feedback.
This work was supported by a CRM-ISM and CIRGET postdoctoral fellowship, by the Fondation Courtois and by the Radboud Excellence Initiative.

\section{Preliminaries}

\subsection{Notations}

Throughout this note $(M,\xi)$ will be a (not necessarily closed) contact manifold of dimension $2n+1$ for some $n\geq 1$.
We assume that $(M,\xi)$ is cooriented, i.e., that $\xi$ is the kernel of some global $1$-form $\alpha$ such that $\alpha\wedge d\alpha^n$ is a volume form on $M$.

We denote by $\D=\D(M):=\diff_0(M)$ the identity component of the group of $C^{\infty}$-diffeomorphisms of $M$ and by $\C=\C(M):=\cont_0(M,\xi)\subset\D$ the identity component of the group of contactomorphisms, i.e., the group of diffeomorphisms preserving $\xi$ that are connected to the identity map via contactomorphisms.

Further denote with $\D_c$ and $\C_c$ the group of compactly supported diffeomorphisms (contactomorphisms) that are connected to the identity via compactly supported diffeomorphisms (contactomorphisms).

A smooth path of diffeomorphisms $(f_t)_{t\in I}$ is a family of diffeomorphisms which is smooth as a map $I\times M\rightarrow M$.
Given a possibly time dependent vector field $X_t$, we will denote its flow at time $t$ by $\phi_t^X$ and the time dependent  vector field generated by a path of diffeomorphisms $f_t$ by $X_t^{f}$.

In the proofs we will often use the following combination of \textbf{concatenation and right translation} of two paths of diffeomorphisms starting at the identity to obtain a path from the identity to the composition of their endpoints.
Given two paths of diffeomorphisms $(f_t)_{t\in [0,1]}$ and $(g_t)_{t\in [0,1]}$ starting at the identity one can obtain a smooth path $(h_t)_{t\in [0,1]}$ from the identity to $f_1g_1$ as follows.
Let $\tau\colon [0,1]\rightarrow\ [0,1]$ be an increasing smooth function with $\tau(0)=0$, $\tau(1)=1$ and $\mathrm{supp}(\tau')\subset [1/2,1]$ and let $\eta\colon [0,1]\rightarrow\ [0,1]$ be an increasing smooth function with $\eta(0)=0$, $\eta(1)=1$ and $\mathrm{supp}(\eta')\subset [0,1/2]$.
Then the path $(h_t)_{t\in [0,1]}=(f_{\tau(t)}g_{\eta(t)})_{t\in[0,1]}$ is smooth and generates the time dependent vector field
$$X_t^h=\tau'(t)X_{\tau(t)}^f\circ g_{\eta(t)}+\eta'(t)df_{\tau(t)}\left(X_{\eta(t)}^g\right).$$
Consequently, if $(f_t)_{t\in [0,1]}$ and $(g_t)_{t\in [0,1]}$ are both null paths, then so is $(h_t)_{t\in [0,1]}$ by the definition of the functions $\tau$ and $\eta$.
Note that this is in general not true for the composition $(f_{t}g_{t})_{t\in[0,1]}$ unless $f_t$ preserves the contact structure.

By a \textbf{Darboux chart} we mean an open subset of $(M,\xi)$ that is contactomorphic to an open subset of $\R^{2n+1}$ equipped with the standard contact structure.
The Darboux Theorem \cite{Geiges} states that any point in a contact manifold has a Darboux neighbourhood, in particular any contact manifold can be covered by Darboux balls.

\subsection{Topologies on groups of diffeomorphisms}\label{sec:top}

We equip the group $\D$ with the (strong) Whitney $C^{\infty}$-topology, see \cite[Chapter 2]{Hirsch76} for details.
Note that on non-compact manifolds this topology is finer than the (weak) compact open $C^{\infty}$-topology, whereas on compact manifolds they coincide.

An important property of the Whitney topology used in the proofs is that the set of diffeomorphisms is open in the space of all smooth maps from $M$ to $M$.
This is even true in the Whitney $C^1$-topology, see \cite[Chapter 2, Theorem 1.6]{Hirsch76}.
For non-compactly supported diffeomorphisms this topology is not locally path-connected and one would need a finer topology to have this property or, for instance give $\D$ the structure of an infinite dimensional manifold, see e.g. \cite{Michor80}.
In particular in the Whitney $C^{\infty}$-topology smooth paths of diffeomorphisms are in general not continuous curves in $\mathcal{D}$ ( in fact the path connected component of $\id$ is contained in $\mathcal{D}_c$).
This makes it impossible to apply the Fragmentation Lemma \cite{Banyaga13} and is the main reason why we restrict to $\R^{2n+1}$ in Theorem \ref{thm:Rn:all}.

We equip the group of compactly supported diffeomorphisms $\D_c$ with the direct limit topology induced by the $C^{\infty}$ compact open topology on the family of subgroups $\D_c(K)$ of diffeomorphisms supported in compact subsets $K\subset M$, see \cite{Michor78, Banyaga13}.
This topology gives rise to the structure of a Lie group on $\D_c$, in particular it is locally contractible \cite{Michor80, Banyaga13}.

\subsection{A Chow-Rashevsky Theorem for diffeomorphisms}

The classical Chow-Rashevsky Theorem in sub-Riemannian geometry states that on a manifold with a bracket generating distribution, i.e., a distribution that is near each point spanned by vector fields whose iterations of Lie-brackets span the whole tangent space, every two points are connected by a curve tangent to the distribution, see e.g. \cite{Gromov96}.
An important example of bracket generating distributions are contact structures.

A distribution $\xi$ on a manifold $M$ naturally defines a distribution $\Xi$ on the group of contactomorphisms by looking at the right invariant set of vector fields
$$\Xi_{\id}:=\{X\in\mathfrak{X}(M)|X_p\in \xi_p \text{ for all } p\in M\}.$$

In \cite[Conjecture 3.5]{Khesin09} the authors conjecture that the Chow-Rashevsky Theorem also holds for paths of diffeomorphisms tangent to the distribution $\Xi$ if the distribution $\xi$ is bracket generating.
This conjecture was proven for compact manifolds in \cite[Corollary 3.1]{Agrachev09}.
In particular the authors show the following.

\begin{theorem}[\cite{Agrachev09}]\label{thm:Chow}
Let $\xi$ be a bracket generating distribution on a smooth compact manifold $M$.
Then for every $f\in\D(M)$ there exist $X_1,\cdots, X_k\in \Xi_{\id}$ such that 
$$f=\phi_1^{X_1}\circ\cdots\circ\phi_1^{X_k}.$$
\end{theorem}

Theorem \ref{thm:Chow} in particular holds in the case when $\xi$ is a contact distribution.

A key step in the proof is the following proposition, which provides some local version of the Chow-Rashevsky theorem.

\begin{prop}[\cite{Agrachev09}]\label{prop:Rn:AC}
Let $X_1,\cdots, X_n$ be vector fields on $\R^n$ such that the $X_i(0)$ span $T_0
\R^n$.
Then there exists a compact neighbourhood $V$ of the origin and an open subset of the space of smooth functions $\mathcal{V}\subset \{f\colon V\rightarrow \R^n|f(0)=0\}$ such that every $f\in\mathcal{V}$ can be written as 
$$f=\phi_1^{a_1X_1}\circ\cdots\circ\phi_1^{a_nX_n}|_V$$
for some smooth functions $a_i\colon \R^n\rightarrow \R$ fixing the origin.
\end{prop}

\section{Proofs}

\subsection{Diffeomorphisms of Darboux-balls}

A simpler version of the above Chow-Rashevsky Theorem which is sufficient for our purposes can be proved more explicitly in the following case.
For $n\geq 1$ consider $\R^{2n+1}$ with its standard contact structure $\xi_0$ given by the kernel of the $1$-form
$$\alpha_0=dz+\sum\limits_{i=1}^nx_idy_i.$$
The proof of the following theorem uses a variation of Proposition \ref{prop:Rn:AC} for the case of the standard contact structure.
The main difference here is that for the standard contact structure the above vector fields used in Proposition \ref{prop:Rn:AC} can be chosen linearly independent on all of $\R^{2n+1}$.
This allows to extend the result to non-compactly supported diffeomorphisms.

\begin{prop}\label{prop:Rn}
For any neighbourhood $\mathcal{U}$ of $\mathrm{id}$ in $\mathcal{D}(\R^{2n+1})$ there exists a neighbourhood $\mathcal{V}$ of $\mathrm{id}$ such that any $f\in \mathcal{V}$ is connected to $\mathrm{id}$ by a null path contained in $\mathcal{U}$.
\end{prop}

\begin{proof}

As the computations work similarly in any dimension, we assume for simplicity that $n=1$.

The tangent space at any point in $\R^3$ is spanned by the vector fields
\begin{align*}
X&:= \partial_x\\
Y&:= \partial_y-x\partial_z\\
Z&=Z_\epsilon:=\partial_y+(x-\epsilon)\partial_z,
\end{align*}
where $\epsilon>0$ is fixed.
The vector fields are chosen such that $X$ and $Y$ span the contact distribution $\xi_0$ and $Z=(\phi_\epsilon^X)_{\ast}Y$, i.e., $\phi_t^Z=\phi_{\epsilon}^X\phi_t^Y(\phi_{\epsilon}^X)^{-1}$.
Their flows are given by
\begin{align*}
\phi_t^X(x,y,z)&=(x+t,y,z)\\
\phi_t^Y(x,y,z)&=(x,y+t,z-xt) \\
\phi_t^Z(x,y,z)&=(x,y+t,z+(\epsilon-x)t).
\end{align*}
Let $f=(f_x,f_y,f_z)\in\D$.
In a first step we want to find functions $\tau_i\colon \R^3\rightarrow \R$, depending on $f$ and $\epsilon$, such that
$$f(x,y,z)=\phi_{\tau_1(x,y,z)}^X\left(\phi_{\tau_2(x,y,z)}^Y\left(\phi_{\tau_3(x,y,z)}^Z(x,y,z)\right)\right).$$
A straightforward computation shows that 
\begin{align*}
\tau_1(x,y,z)&=f_x-x\\
\tau_2(x,y,z)&=f_y-y-\frac{f_z-z+x(f_y-y)}{\epsilon} \\
\tau_3(x,y,z)&=\frac{f_z-z+x(f_y-y)}{\epsilon}.
\end{align*}

Note that $f$ is not the composition of the functions $\phi_{\tau_1}^X, \phi_{\tau_2}^Y,\phi_{\tau_3}^Z$, which are not diffeomorphisms in general.
To write $f$ as a composition of (re-parametrized) flows of $X,Y,Z$, we need to be able to invert those maps.

Following the proof of Proposition 4.1 in \cite{Agrachev09} define 
\begin{align*}
\Phi_1(x,y,z)&:=\phi_{\tau_3(x,y,z)}^Z(x,y,z)=(x,y+\tau_3, z+(\epsilon-x)\tau_3)\\
\Phi_2(x,y,z)&:=\phi_{\tau_2(x,y,z)}^Y\left(\phi_{\tau_3(x,y,z)}^Z(x,y,z)\right)=(x,f_y,f_z).
\end{align*}

\textbf{Claim:} For fixed $\epsilon>0$ there exists a neighbourhood $\mathcal{W}_{\epsilon}$ such that the maps $\Phi_1$ and $\Phi_2$ are invertible for all $f\in\mathcal{W}_{\epsilon}$.

The group of diffeomorphisms is $C^1$-open inside the space of all smooth maps, see \cite[Chapter 2, Theorem 1.6]{Hirsch76}.
In particular there exist a locally finite family of compact subsets $(K_i)_{i\in A}$ and a family of positive numbers $(\delta_i)_{i\in A}$ such that if a smooth map $g$ satisfies $|g(p)-p|<\delta_i$ and $|Dg(p)-\mathrm{Id}_{3}|<\delta_i$ for all $i\in A$ and $p\in K_i$ then $g$ is a diffeomorphism.

Indeed one can compute that on some compact subset $K_i$
\begin{align*}
&|\Phi_1(x,y,z)-(x,y,z)|=\left\vert\left(0,\frac{f_z-z+x(f_y-y)}{\epsilon},(\epsilon-x)\frac{f_z-z+x(f_y-y)}{\epsilon}\right)\right\vert\\
&\leq \frac{1}{\epsilon}(\left\vert(0,f_z-z,(\epsilon-x)(f_z-z))\right\vert+|x||(0,f_y-y,(\epsilon-x)(f_y-y))|)\\
&\leq C_1\max\limits_{p\in K_i}|f(p)-p|,
\end{align*}
where $C_1(K_i, \epsilon)$ is a constant depending on $K_i$ and $\epsilon$.

Similarly one shows that 
$$|D\Phi_1(p)-\mathrm{Id}_{3}|\leq C_2\max\limits_{p\in K_i}|Df(p)-\mathrm{Id}_{3}|$$ 
for some constant $C_2(K_i,\epsilon)$.
Hence choosing $f$ close enough to the identity we have that $|\Phi_1(p)-p|<\delta_i$ and $|D\Phi_1(p)-\mathrm{Id}_{3}|<\delta_i$ for all $i\in A$ and $p\in K_i$.
Obviously the same is true for $\Phi_2$, which proves the claim.

It follows that 
\begin{align}\label{eq1}
f=\phi^X_{\tau_1\circ\Phi_2^{-1}}\circ\phi^Y_{\tau_2\circ \Phi_1^{-1}}\circ\phi^Z_{\tau_3}=\phi^X_{\tau_1\circ\Phi_2^{-1}}\circ\phi^Y_{\tau_2\circ \Phi_1^{-1}}\circ\phi^X_{\epsilon}\circ\phi^Y_{\tau_3\circ \phi_{\epsilon}^X}\circ(\phi^X_{\epsilon})^{-1}
\end{align}
for $f\in \mathcal{W}_{\epsilon}$.

Let $\mathcal{U}$ be any open neighbourhood of $\mathrm{id}$ in the space of smooth maps.
A similar straightforward estimate using the above explicit expressions shows that one can pick the neighbourhood $\mathcal{W}_{\epsilon}$ of $\mathrm{id}$ in $\D$ such that the maps $\phi^X_{s(\tau_1\circ\Phi_2^{-1})}$, $\phi^Y_{s(\tau_2\circ \Phi_1^{-1})}$ and $\phi^Y_{s(\tau_3\circ \phi_{\epsilon}^X)}$ are well-defined diffeomorphisms and lie in $\mathcal{U}$ for every fixed $f\in \mathcal{W}_{\epsilon}$ and $s\in[0,1]$.
Moreover for $\epsilon$ sufficiently small the paths $(\phi_{\pm s\epsilon}^X)_{s\in[0,1]}$ are contained in $\mathcal{U}$.

The paths $\phi^X_{s(\tau_1\circ\Phi_2^{-1})}$, $\phi^Y_{s(\tau_2\circ \Phi_1^{-1})}$ and $\phi^Y_{s(\tau_3\circ \phi_{\epsilon}^X)}$ are generated by time-dependent vector fields tangent to $\xi_0$.
Their concatenations together with concatenating $(\phi_{\pm s\epsilon}^X)_{s\in[0,1]}$ defines a null path from $\mathrm{id}$ to $f\in\mathcal{W}_{\epsilon}$. 
After possibly diminishing $\mathcal{V}:=\mathcal{W}_{\epsilon}$, we can assume that this path is contained in $\mathcal{U}$ for all $f\in\mathcal{V}$.

\end{proof}

\subsection{Proof of Theorem \ref{thm:Rn:all}}

Since $\mathcal{D}(\R^{2n+1})$ equipped with the Whitney $C^{\infty}$-topology is a topological group, see e.g. \cite{Hirsch76}, the group is generated by any open neighbourhood of the identity.
Proposition \ref{prop:Rn} implies the existence of a neighbourhood $\mathcal{U}$ around the identity such that any $g\in\mathcal{U}$ is connected to $\id$ via a null path of diffeomorphisms.

Every diffeomorphisms is a finite composition $f=f_1\circ\cdots\circ f_k$, where $f_i\in\mathcal{U}$.
By concatenations and right translation of the paths connecting the identity to the $f_i$ we obtain a null path from $\id$ to $f$.

Let $\phi_t$ be the Reeb-flow of some contact form $\alpha$, i.e., $\phi_t$ is the flow of a vector field $R$ uniquely defined by the equations $\alpha(R)=1$ and $d\alpha(R,\cdot)=0$.
Then $\phi_t$ is a positive path of diffeomorphisms preserving the contact form $\alpha$.
Pick a null path of diffeomorphisms $(g_t)_{t\in[0,1]}$ with $g_0=\mathrm{id}$ and $g_1=\phi_{-1}f$.
The composition $(\phi_tg_t)_{t\in[0,1]}$ is a positive path of diffeomorphisms from $\mathrm{id}$ to $f$.

Every $f\in \D_c$ is the composition of finitely many elements $f_1,\cdots, f_k$ in $\D_c\cap\mathcal{U}$.
By construction of the null paths in Proposition \ref{prop:Rn}, every $f_i$ is connected to $\id$ by a compactly supported null path.
Hence the same is true for $f$.

Consider the standard Reeb flow on $\R^{2n+1}$, which in dimension $3$ is given by $\phi_t(x,y,z)=(x,y,z+t)$.
The null paths to $\phi_t$ for fixed $t$ obtained in the proof of Proposition \ref{prop:Rn} can be computed explicitly as
\begin{align}
\phi^X_{s(\tau_1\circ\Phi_2^{-1})}&=\id\\
\phi^Y_{s(\tau_2\circ\Phi_1^{-1})}&=\left(x,y-\frac{st}{\epsilon},z-(\epsilon-x)\frac{st}{\epsilon}\right)\\
\phi^Y_{s(\tau_3\circ\phi_{\epsilon}^{X})}&=\left(x,y+\frac{st}{\epsilon},z+(\epsilon-x)\frac{st}{\epsilon}\right).
\end{align}\label{eq2}
The higher dimensional case works analogously.

Let $\mathcal{U}_{co}$ be a neighbourhood of the identity in the $C^{\infty}$-compact open topology.
For $|t|$ sufficiently small the Reeb flow $\phi_t$ at time $t$ lies in $\mathcal{U}_{co}$ and from the above expressions one sees that, picking $|t|$ small enough, $\phi_t$ is connected to $\id$ by a null path contained in $\mathcal{U}_{co}$.
Given a compactly supported path $(f_s)_{s\in[0,1]}$ and fixed $t$ small enough one gets a null path $g_s$ connecting $\id$ and $\phi_{-t}\circ f_1$ that is arbitrarily close in the compact open topology to the identity outside of some compact subset.
As before we get the desired positive path by looking at the composition $\phi_{st}\circ g_s$.

\subsection{Proof of Theorem \ref{thm:extension}}

Let $(f_t)_{t\in [0,1]}$ be a path of contactomorphisms that is positive outside a compact set $K_0$.
We can without restriction assume that $f_0=\id$.
Pick a compact set $K_1$ such that $f^{-1}_t(K_0)\subset K_1$ for all $t\in [0,1]$.
Given a contact form $\alpha$, let $C_1$ be the minimum of the time-dependent contact Hamiltonian inducing the path $(f_t)_{t\in [0,1]}$.
Note that $f_t^{\ast}\alpha=\rho_t\alpha$ for some family of positive functions $(\rho_t)_{t\in[0,1]}$.
Let $C_2$ be the maximum of the family $(\rho_t)_{t\in[0,1]}$ on $K_1$.
Pick compact subsets $K_2, K_3$ with $K_1\subset K_2\subset K_3$ and a non-negative function $h$ with $h= -C_1C_2+1$ on $K_2$ and $h=0$ on $\R^{2n+1}\setminus K_3$.
Let $(\phi_t)_{t\in [0,1]}$ be the contact flow generated by the contact Hamiltonian $h$.
Then by Theorem \ref{thm:Rn:all} there exists a compactly supported null path $(\psi_t)_{t\in [0,1]}$ connecting the identity and $\phi_{-1}$.
We pick the sets $K_2$ and $K_3$ such that $\phi_t(K_1)\subset K_2$ for all $t\in [-1,1]$.
Hence on $K_2$ the flow $(\phi_t)_{t\in [0,1]}$ coincides with a reparametrization of the standard Reeb flow.
Moreover, from the expressions (2)-(4) for the null path in this case, one sees that if we pick $K_2$ and $K_3$ large enough we can assume that $\phi_t(\psi_t(K_1))\subset K_2$ for all $t\in [0,1]$.
Then $(g_t)_{t\in [0,1]}$ with $g_t=\phi_t\circ\psi_t$ is a positive loop based at $\id$ with $\alpha\left(X_t^g\right)>-C_1C_2$ on $K_1$.

It follows that $(f_t\circ g_t)_{t\in [0,1]}$ is a positive path of diffeomorphisms between $\id$ and $f_1$ that coincides with $f_t$ outside a compact subset for $t\in [0,1]$.
 
\subsection{Proof of Theorem \ref{thm:path}}

Let $f\in\D_c(M)$ be compactly supported in some compact set $K$ let $U_1,\cdots, U_m$ be a finite cover of $M$ by Darboux charts.
The fragmentation Lemma \cite{Banyaga13} implies that there exist $f_1,\cdots, f_j$ such that $f=f_1\circ\cdots\circ f_j$ and $f_i\in\D_c(U_{k(i)})$.
It follows from Theorem \ref{thm:Rn:all} that there exist null paths from $\id$ to $f_i$.
Note that the paths constructed in the proof of Theorem \ref{thm:Rn:all} are compactly supported in $U_{k(i)}$ if $f_i$ is compactly supported.
By concatenating and right translating these paths we obtain a compactly supported null path from the identity to $f$.

Let $\psi_1,\cdots, \psi_m$ a partition of unity subordinate to the cover $U_1,\cdots, U_m$.
Given a contact form $\alpha$ consider the flow $\phi_t^{i}$ induced by the contact Hamiltonian $\psi_i$, i.e., $\phi_t^{i}$ is the flow of the unique contact vector field $X$ with $\alpha(X)=\psi_i$, see \cite{Geiges} for details on contact Hamiltonians.
The composition $\phi_t:=\phi_t^1\circ\cdots\circ\phi_t^m$ is a positive path of contactomorphisms.
Moreover, for fixed $t$ the diffeomorphism $\phi_t$ is connected to $\id$ by a null path since by Theorem \ref{thm:Rn:all} each of the $\phi_t^i$ is.
Hence given a compactly supported path $(f_t)_{t\in[0,1]}$, the diffeomorphism $\phi_{-1}f_1$ is connected to the identity by a null path $(g_t)_{t\in [0,1]}$.
The desired positive path from $\id$ to $f_1$ is given by $(\phi_t\circ g_t)_{t\in [0,1]}$.
 
\subsection{Proof of Theorem \ref{thm:homotopic}}

Let $(f_t)_{t\in[0,1]}$ be a compactly supported path of diffeomorphisms. 
In the limit topology chosen on $\D_c$, see \ref{sec:top}, $(f_t)_{t\in[0,1]}$ defines a continuous path.
As this topology is locally contractible, we can pick a finite cover $\mathcal{U}_1,\cdots, \mathcal{U}_m$ of $(f_t)_{t\in[0,1]}$ by contractible open subsets.
Pick $t_0,\cdots, t_l\in [0,1]$ such that $t_0=0, t_l=1$ and $t_i$ and $t_{i+1}$ are contained in the same contractible open subset $\mathcal{U}_{k(i)}$.
We will first show that $f_{t_i}$ and $f_{t_{i+1}}$ can be connected in $\mathcal{U}_{k(i)}$ by a null path of diffeomorphisms.
Concatenating these paths we obtain a null path that is homotopic to $(f_t)_{t\in[0,1]}$.

\begin{lemma}
Let $(g_t)_{t\in[0,1]}$ be a path of diffeomorphisms contained in an open subset $\mathcal{U}\subset \D_c$.
Then there exists a null path of diffeomorphisms contained in $\mathcal{U}$ connecting $g_0$ and $g_1$.
\end{lemma}

\begin{proof}
It is no restriction to assume that $f_0=\id$.
Let $K$ be a compact subset such that $(g_t)_{t\in[0,1]}$ is supported in $K$ and let $U_1,\cdots, U_k$ be an open cover of $K$ by Darboux charts.
The fragmentation Lemma \cite{Banyaga13} implies that $g_t=g_t^1\circ\cdots\circ g_t^l$, where $g_t^i$ is contained in $\D_c(U_{l(i)})\cap\mathcal{U}$.

By Proposition \ref{prop:Rn} for every fixed neighbourhood $\mathcal{W}_i\subset \D_c(U_{l(i)})\cap\mathcal{U}$ and each $t\in [0,1]$ there exists a neighbourhood $\mathcal{V}_t$ of $g_t$ in $\D_c(U_{l(i)})$ such that each $h\in \mathcal{V}_t$ is connected to $g_t$ by a null path contained in $\mathcal{W}_i$.
Pick a subdivision $0=t_0\leq \cdots \leq t_m=1$ such that $g_{t_{j+1}}\in \mathcal{V}_{t_j}$ for all $0\leq j\leq m$.
Then $g_{t_j}$ and $g_{t_{j+1}}$ are connected by a null path contained in $\mathcal{W}_i$.
The concatenation of these paths is a null path between $\id$ and $g_t^i$ in $\mathcal{W}_i$.

Choosing each $\mathcal{W}_i$ small enough and concatenating and right translating each of these null paths, we obtain a null path connecting $\id$ and $g_1$ contained in $\mathcal{U}$.

\end{proof}

The Lemma also implies the existence of small null and positive loops.

In case $M$ is closed, all the topologies described above coincide, in particular the Whitney $C^{\infty}$-topology is locally contractible on $\D(M)$.
As before let $\phi_t$ be the Reeb flow of some contact form.
Left composing each of the null paths obtained above with $(\phi_{\epsilon t})_{t\in[0,1]}$ we obtain positive paths between $f_{t_i}$ and $f_{t_{i+1}}$ contained in $\mathcal{U}_{k(i)}$.
Their concatenation is a positive path homotopic to $(f_t)_{t\in[0,1]}$.

\subsection{Proof of Theorem \ref{cor:positive:contraction}} 

Let $U_1\cdots U_k$ be an open cover of $M$ by Darboux charts and $\psi_1,\cdots, \psi_k$ a subordinate partition of unity.
Given a contact form $\alpha$ and $\delta>0$ consider the flow $\phi_t^{i,\delta}$ induced by the contact Hamiltonian $\delta\psi_i$, i.e., $\phi_t^{i,\delta}$ is the flow of the unique contact vector field $X$ with $\alpha(X)=\delta\psi_i$.
The map $\phi_t^{i,\epsilon}$ is compactly supported in $U_i$.
Moreover, the composition $\phi_t^{\delta}:=\phi_t^{1,\delta}\circ\cdots\circ \phi_t^{k,\delta}$ defines a positive path of contactomorphisms.

For $\delta$ sufficiently small, all $\phi_{-s}^{i,\delta}$ are connected to $\id$ by a null path $(f_t^{s,i})_{t\in[0,1]}$ of the form (\ref{eq1}) for each $s\in[0,1]$.
Their concatenation and right translation yields a null path $(f_t^s)_{t\in[0,1]}$.
From the construction of the $f_t^{s,i}$ one can see that $(f_t^s)_{t\in[0,1]}$ depends continuously on $s$.
Then $\phi_{st}^{\delta}\circ f_t^s$ is a homotopy of positive loops.

\end{document}